\documentclass[twoside,12pt]{article}
\usepackage{amsthm,amsmath,amssymb,amscd,enumerate,epsfig}
\usepackage{amsfonts}
\usepackage{graphicx}
\usepackage{fancyhdr}
\begin{document}
\textheight=570pt 
\begin{center} {\Large\bf On the Non-negative Integer Solutions to Diophantine Equations ${\bf F_n - F_m = 7^a}$ and 
${\bf F_n - F_m = 13^a}$}
\vskip 1cm {\bf Gaha Anouar, Mezroui Soufiane}\\
\
\\Mathematics and Intelligent Systems (MASI)\\
National School of Applied Sciences of Tangier (ENSAT)\\
Abdelmalek Essaadi University\\
Tangier, Morocco\\
\
\\email: anouar.louis.gaha@gmail.com\\
\end{center}
\begin{center}
(Received Insert here the date of submitting your paper to the journal)
\end{center}
\begin{abstract}
In this paper, we study the solutions of the equation $F_n-F_m=p^a$ where $p$ is either $7$ or $13$ and $n>m\geqslant 0$, 
$a\geqslant 2$. We confirm the conjecture of Erduvan and Keskin by proving that there is no solutions for this Diophantine equation. We will use the lower bounds for linear forms in logarithms (Baker's theory) and a version of the Baker-Davenport reduction method in Diophantine approximation.
\end{abstract}
\begingroup
\newtheorem{theorem}{Theorem}[section]
\newtheorem{lemma}[theorem]{Lemma}
\newtheorem{proposition}[theorem]{Proposition}
\newtheorem{corollary}[theorem]{Corollary}
\newtheorem{definition}[theorem]{Definition}
\newtheorem{remark}[theorem]{Remark}
\endgroup
\def\E#1{\left\langle #1 \right\rangle}
\def\diag{\mathop{\fam 0\relax diag}\nolimits}
\section{Introduction}
The sequence of Fibonacci numbers $(F_n)_{n\geqslant 0}$ is defined recursively as $F_0=0$, $F_1=1$, and $F_n=F_{n-1}+F_{n-2}$ for all $n\geqslant 2$. The sequence of Lucas numbers $(L_n)_{n\geqslant 0}$ is similarly defined as $L_0=2$, $L_1=1$, and $L_n=L_{n-1}+L_{n-2}$ for all $n\geqslant 2$. The terms of the Fibonacci and Lucas sequences are called Fibonacci and Lucas numbers, respectively. The Fibonacci and Lucas numbers for negative indices are defined by $F_{-n}=(-1)^{n+1}F_n$ and
$L_{-n}=(-1)^{n}L_n$ for $n\geqslant 1$. 

The Fibonacci and Lucas sequences have many interesting properties and have been studied in the literature by many researchers \cite{LD}. In 2006, Bugeaud, Mignotte, and Siksek \cite{BMS6} studied the problem of finding all perfect powers in the Fibonacci sequence and the Lucas sequence by combining a classical techniques with the modular approach. In 2014, Bravo and Luca \cite{BL14} investigated the solutions of the Diophantine equation $L_n+L_m=2^a$ in non-negative integers $n$, $m$, and $a$. The Diophantine equation $F_n \pm F_m=y^a$ where $n\geqslant m\geqslant 0$, $y\geqslant 2$, and $a\geqslant 2$ has been studied by a number of the authors (see \cite{BPS,BMS6,BLMS8} for the cases $m \in \{0,1,2,3\}$). In 2018, Luca and Patel \cite{LP} showed that this Diophantine equation has solution which is either $\max \{|n|,|m|\} \leqslant 36$ or $y=0$ and $|n|=|m|$ if $n\equiv m\pmod 2$. But the case $n\not \equiv m\pmod 2$ remained an open problem. In 2019, Demirtürk and Keskin \cite{DK} found all the solutions with $y=3$ and $n\not \equiv m\pmod 2$. In the same year, Erduvan and Keskin \cite{EK} found all the solutions with $y=5$ and $n\not \equiv m\pmod 2$. The authors conjectured that the equation $F_n - F_m=y^a$ has no solutions in non-negative integers $m$ and $n$ when $y>7$ is prime and $a\geqslant 2$. In 2020, \c{S}iar and Keskin \cite{SK} studied the problem for the equation $F_n - F_m=2^a$.  

In this paper, we study the solutions in non-negative integers to the Diophantine equation 
\begin{equation} \label{eq1}
F_n - F_m=p^a 
\end{equation}
where $n>m\geqslant 0$, $p=7,13$ and $a\geqslant 2$. Specifically, we demonstrate the following two theorems.

\begin{theorem} \label{thm3}
The only solutions of the Diophantine equation $F_n - F_m=7^a$ in non-negative integers $(n,m,a)$ with $n \geqslant 2$ and $m<n$, are given by
$$(n,m,a) \in \{(2,0,0),(3,1,0),(3,2,0),(4,3,0),(6,1,1),(6,2,1)\},$$
namely
$$F_2 - F_0 = F_3 - F_1 = F_3 - F_2 = F_4 - F_3 = 7^0$$ and $$F_6 - F_1 = F_6 - F_2 = 7^1.$$
\end{theorem}

\begin{theorem} \label{thm4}
The only solutions of the Diophantine equation $F_n - F_m=13^a$ in non-negative integers $(n,m,a)$ with $n \geqslant 2$ and $m<n$, are given by
$$(n,m,a) \in \{(2,0,0),(3,1,0),(3,2,0),(4,3,0),(7,0,1),(8,6,1),(9,8,1)\},$$
namely
$$F_2 - F_0 = F_3 - F_1 = F_3 - F_2 = F_4 - F_3 = 13^0$$ and $$F_7 - F_0=F_8 - F_6 = F_9 - F_8 = 13^1.$$
\end{theorem}

\section{Preliminary results}
Let $\eta$ be an algebraic number of degree $d$ with minimal polynomial
$$\sum_{\substack{i=0}}^{d}a_i X^{d-i}=a_0\prod_{\substack{i=1}}^{d}\left(X-\eta^{(i)}\right)\in \mathbb{Z}[X],$$
where the $a_i$'s are relatively prime, $a_0>0$ and the $\eta^{(i)}$'s are the conjugates of $\eta$. The logarithmic height of
$\eta$ is given by
\begin{equation}\label{eq3}
h(\eta):=\frac{1}{d}\left(\log a_0+\sum_{\substack{i=1}}^{d}\log \left(\max\{|\eta^{(i)}|,1\}\right)\right).
\end{equation} 
In particular, if $\eta=\frac{p}{q}\in \mathbb{Q}$ where $p, q>0$ are relatively prime integers, then 
$h(\eta)= \log \left(\max\{|p|,|q|\}\right)$. The logarithmic height has the following well-known properties
\begin{equation}\label{eq4}
h(\eta_1\pm\eta_2)\leqslant h(\eta_1)+h(\eta_2)+\log(2);
\end{equation}
\begin{equation}\label{eq5}
h(\eta_1\eta_{2}^{\pm1})\leqslant h(\eta_1)+h(\eta_2);
\end{equation}
\begin{equation}\label{eq6}
h(\eta^k)=|k| h(\eta), \quad k\in \mathbb{Z}.
\end{equation}

The following theorem can be deduced from Corollary 2.3 of Matveev \cite{EMM}.
\begin{theorem} \label{thm1}
(Matveev’s theorem) Assume that $\eta_1,\eta_2,\ldots,\eta_t$ are positive real algebraic numbers in a real algebraic number field $\mathbb{K}\subset \mathbb{R}$ of degree $D$. Let $b_1,b_2,\ldots,b_t$ are rational integers such that 
$$\Gamma:=\eta_{1}^{b_1} \ldots \eta_{t}^{b_t}-1 \neq 0.$$
Then $$|\Gamma|>\exp\left(-1.4\cdot 30^{t+3}\cdot t^{4.5}\cdot D^2(1+\log(D))(1+\log(B))A_1A_2\ldots A_t\right),$$
where $$B\geqslant \max\{|b_1|,\ldots,|b_t|\},$$
and $$A_i \geqslant \max\{Dh(\eta_i),|\log(\eta_i)|,0.16\} \quad \text{for all}\quad i=1,\ldots,t.$$
\end{theorem}
The following result follows from Lemma 5 of Dujella and Peth\H{o} \cite{DP} providing a variant of Baker-Davenport reduction method. For $x\in \mathbb{R}$, we write $\|x\|:=\min\{|x-n|: n\in \mathbb{Z}\}$, where $\|\cdot\|$ denotes the distance from $x$ to the nearest integer.
 
\begin{lemma} \cite{DP} \label{lem1}
(Dujella and Peth\H{o}) Let $M>1$ be a positive integer and suppose that $\frac{p}{q}$ is a convergent of the continued fraction expansion of the irrational number $\gamma$ such that $q>6M$, and let $A$, $B$, $\mu$ be some real numbers with $A>0$ and $B>1$. Let $\varepsilon:=\|\mu q\|-M\|\gamma q\|$. If $\varepsilon>0$, then there exists no solution to the inequality
$$0<|m\gamma-n+\mu|<AB^{-\omega},$$ in positive integers $m$, $n$ and $\omega$ with $m\leqslant M$ and 
$$\omega\geqslant\frac{\log(Aq/\varepsilon)}{\log(B)}\cdot$$
\end{lemma}

The following result is obtained from Theorem 1 and 2 of Bugeaud, Migonette, and Siksek \cite{BMS6}.
\begin{theorem} \label{thm2}
The only perfect powers in the Fibonacci sequence are $F_0=0$, $F_1=F_2=1$, $F_6=8$, and $F_{12}=144$. The only perfect powers in the Lucas sequence are $L_1=1$ and $L_3=4$.
\end{theorem}

The following result is extracted from Lemma 2.1 of Luca and Patel \cite{LP}.
\begin{lemma} \label{lem2}
Assume that $n\equiv m\pmod 2$. Then
\begin{equation*}
F_n-F_m=\\
\begin{cases}
F_{(n-m)/2}L_{(n+m)/2} \quad \text{if} \quad n\equiv m\pmod 4,\\
F_{(n+m)/2}L_{(n-m)/2} \quad \text{if} \quad n\equiv m+2\pmod 4.
\end{cases} 
\end{equation*}
\end{lemma}

\section{Solutions to the equation (\ref{eq1}) for $p=7$}
The $n$th Fibonacci number can be represented in the form 
\begin{equation}\label{eq7}
F_n=\frac{\alpha^n-\beta^n}{\sqrt{5}} \quad \text{for all} \quad n\geqslant 0,
\end{equation}
where $\alpha:=\frac{1+\sqrt{5}}{2}$ and $\beta:=\frac{1-\sqrt{5}}{2}$. The following inequality
\begin{equation}\label{eq8}
\alpha^{n-2}\leqslant F_n \leqslant \alpha^{n-1}
\end{equation}
is true for all $n\geqslant 1$ and it can be verified by induction on $n$.

\begin{proof}[Proof of Theorem \ref{thm3}]
Suppose there are non-negative integers $n$, $m$, $a$ with $n\geq 2$ and $m<n$ satisfying $F_n - F_m=7^a$. By using the identity (\ref{eq8}), we obtain the inequality
\begin{equation} \label{eq9}
7^a=F_n - F_m<F_n \leqslant \alpha^{n-1}< \alpha^{n}.
\end{equation}
This shows that $a<n$.\\
Suppose $n-m=1$. Then $F_{m-1}=7^a$. By Theorem \ref{thm2}, we have $(n,m,a) \in \{(3,2,0),(4,3,0)\}$. Equally, in the case 
$n-m = 2$, then we get $F_{m+1}=7^a$. We have $(n,m,a) \in \{(2,0,0),(3,1,0)\}$. Let $m = 0$. By Theorem \ref{thm2}, we have that solution $F_2 - F_0 = 1$. Assume that $m \geqslant 1$ and $n-m \geqslant 3$. From $1 \leqslant m < n \leqslant 200$, a direct calculation in Maple program, we obtain the only solutions are $(n,m,a) \in \{(6,1,1),(6,2,1)\}$.\\
Assume then that $n>200$, $m \geqslant 1$ and $n-m \geqslant 3$. The formula (\ref{eq7}) can be rewritten as
$$\frac{\alpha^{n}}{\sqrt{5}} - 7^a = F_m + \frac{\beta^{n}}{\sqrt{5}},$$
then
$$\left| \frac{\alpha^{n}}{\sqrt{5}} - 7^a \right| = \left| F_m +  \frac{\beta^{n}}{\sqrt{5}} \right| \leqslant F_m + \frac{\mid \beta \mid^{n}}{\sqrt{5}} < \alpha^m + \frac{1}{2}\cdot$$
Dividing both sides of the above inequality by $\frac{\alpha^{n}}{\sqrt{5}}$ and taking into account that $n>m$, we have 
\begin{equation} \label{eq10}
\left| 1 - 7^a \alpha^{-n}\sqrt{5} \right| < \alpha^{m-n}\sqrt{5} + \frac{\alpha^{-n}\sqrt{5}}{2} = \alpha^{m-n}\sqrt{5} 
\left(1 + \frac{1}{2\alpha^{m}} \right)< \frac{4}{\alpha^{n-m}}\cdot
\end{equation}
Now, we apply Theorem \ref{thm1}, we take the parameters $t:=3$ and $\eta_1 := 7$, $\eta_2 := \alpha$, $\eta_3 := \sqrt{5}$. We also take $b_1 := a$, $b_2 := -n$ and $b_3 := 1$. Note that the three numbers $\eta_1$, $\eta_2$, and $\eta_3$ are positive real numbers and elements of the field $\mathbb{K}:=\mathbb{Q}(\sqrt{5})$, so we can take $D:=[\mathbb{K}:\mathbb{Q}]=2$. We show that $\Gamma_1 := 7^a\alpha^{-n}\sqrt{5} - 1 \neq 0$. Assume that $\Gamma_1 = 0$. We obtain $\alpha^{2n} = 5.7^{2a} \in \mathbb{Q}$, which is impossible since $\alpha^{2n} \not \in \mathbb{Q}$. Moreover, since $h(\eta_1)=\log(7)= 1.94591 \ldots$, we take $A_1:=3.9$. Since $h(\eta_2)=\frac{\log(\alpha)}{2}=0.2406\ldots$ and $h(\eta_3)=\log(\sqrt{5})=0.80471\ldots$, we can take $A_2:=0.5$ and $A_3:=1.65$. Since $a<n$, it follows that
$$B:= \max \{|b_1|,|b_2|,|b_3| \} = \max \{|a|,|-n|,1 \} = n.$$
Hence, by inequality (\ref{eq10}) and using Theorem \ref{thm1}, we have
$$\exp\left(-C(1+ \log(2))(1+ \log(n))(3.9\cdot 0.5\cdot 1.65)\right)<|\Gamma_1|<\frac{4}{\alpha^{n-m}}$$
and so 
$$(n-m)\log(\alpha) - \log(4) < C\left(1+ \log(2)\right)\left(1+ \log(n)\right)\left(3.9\cdot 0.5\cdot 1.65\right),$$
where $C=1.4\cdot 30^6 \cdot 3^{4.5}\cdot 2^2$.\\
From the last inequality, we obtain 
\begin{equation} \label{eq11}
(n-m)\log(\alpha) - \log(4)<3.12014\cdot 10^{12}\left(1+ \log(n)\right).
\end{equation}
Now, we again rewrite equation (\ref{eq1}) as
$$\frac{\alpha^n}{\sqrt{5}}-\frac{\alpha^m}{\sqrt{5}}-7^a=\frac{\beta^n}{\sqrt{5}}-\frac{\beta^m}{\sqrt{5}}$$  
and taking absolue values, we obtain
$$\left|\frac{\alpha^n(1-\alpha^{m-n})}{\sqrt{5}}-7^a\right|=\frac{|\beta|^n+|\beta|^m}{\sqrt{5}}<\frac{1}{3},$$
where we used the fact that $|\beta|^n+|\beta|^m<\frac{2}{3}$ for $n>200$. If we divide both sides of the above inequality by  
$\frac{\alpha^n(1-\alpha^{m-n})}{\sqrt{5}}$, we get
\begin{equation} \label{eq12}
\left|1-7^a\alpha^{-n}\sqrt{5}\left(1-\alpha^{m-n}\right)^{-1}\right|<\frac{\alpha^{-n}\sqrt{5}\left(1-\alpha^{m-n}\right)^{-1}}{3}\cdot
\end{equation}
Since $\alpha^{m-n}=\frac{1}{\alpha^{n-m}}<\frac{1}{\alpha}<\frac{2}{3}$, we see that $1-\alpha^{m-n}>\frac{1}{3}$ and hence 
$\left(1-\alpha^{m-n}\right)^{-1}<3$. Then from (\ref{eq12}), it follows that  
\begin{equation} \label{eq13}
\left|1-7^a\alpha^{-n}\sqrt{5}\left(1-\alpha^{m-n}\right)^{-1}\right|<\frac{\sqrt{5}}{\alpha^{n}}<\frac{3}{\alpha^{n}}\cdot
\end{equation}
We apply Theorem \ref{thm1} once again. Take $\eta_1:=7$, $\eta_2:=\alpha$, $\eta_3:=\sqrt{5}(1-\alpha^{m-n})^{-1}$, $b_1:=a$, $b_2:=-n$ and $b_3:=1$. The positive real numbers $\eta_1$, $\eta_2$, and $\eta_3$ lie in $\mathbb{K}:=\mathbb{Q}(\sqrt{5})$, so we have $D=2$. Put $\Gamma_2:=7^a\alpha^{-n}\sqrt{5}(1- \alpha^{m-n})^{-1}-1$. We show that $\Gamma_2$ is not zero. Because if        
$\Gamma_2=0$, then we get
$$\frac{\alpha^n}{\sqrt{5}}-\frac{\alpha^m}{\sqrt{5}}=7^a=F_n-F_m=\frac{\alpha^n}{\sqrt{5}}-\frac{\alpha^m}{\sqrt{5}}+\frac{\beta^m}{\sqrt{5}}-\frac{\beta^n}{\sqrt{5}},$$
which implies that $\beta^n=\beta^m$. Therefore, this is not acceptable since $n>m$. Similarly, by (\ref{eq3}), we have  
$h(\eta_1)=\log(7)=1.94591\ldots$, and $h(\eta_2)=\frac{\log(\alpha)}{2}=0.2406\ldots$, we take $A_1:=3.9$ and $A_2:=0.5$.
Bisides, by equations (\ref{eq4}), (\ref{eq5}), and (\ref{eq6}), we get that  
\begin{center}
$h(\eta_3)\leqslant \log(\sqrt{5})+|m-n|\frac{\log(\alpha)}{2}+\log(2)=\frac{1}{2}\left(\log(20)+(n-m)\log(\alpha)\right)$.
\end{center}
A simple calculation shows that $|\log(\eta_3)|<\log(5)+(n-m)\log(\alpha)$, and so we can take $A_3:=\log(20)+(n-m)\log(\alpha)$. 
If follows again that $B:=\max\{|a|,|-n|,1\}=n$, since $a<n$. Thus, from inequality (\ref{eq13}) and Theorem \ref{thm1}, we obtain
\begin{multline*}
\exp\left(-C(1+\log(2))(1+ \log(n))(3.9\cdot0.5(\log(20)+(n-m)\log(\alpha)))\right)\\
<|\Gamma_2|<\frac{\sqrt{5}}{\alpha^{n}}<\frac{3}{\alpha^{n}}
\end{multline*}
where $C=1.4\cdot 30^6 \cdot 3^{4.5}\cdot 2^2$, which yields
\begin{equation}\label{eq14}
\log(\frac{\alpha^n}{3})<1.89099\cdot 10^{12}\left(1+ \log(n)\right)\left(\log(20)+(n-m)\log(\alpha)\right).
\end{equation} 
We can now substitute inequality (\ref{eq11}), we obtain
\begin{equation*}
\begin{split}
n\log(\alpha)-\log(3)<&1.89099\cdot 10^{12}\left(1+ \log(n)\right)\\
&\cdot\left(\log(20)+3.12014\cdot 10^{12}(1+ \log(n))+\log(4)\right)
\end{split}
\end{equation*}
and so
\begin{equation*}
n<1.22613\cdot10^{25}+2.45226\cdot10^{25}\log(n)+1.22613\cdot10^{25}(\log(n))^2.
\end{equation*}
Then $$n<1.46212\cdot10^{26}\left(\log(n)\right)^2.$$
With the help of Maple program, it is seen that $$n<6.90212\cdot10^{29}.$$
Let us try to reduce the upper bound on $n$ by applying Lemma \ref{lem1} two times. Suppose  
$$\theta_1:=\log(7^a\alpha^{-n}\sqrt{5})=a\log(7)-n\log(\alpha)+\log(\sqrt{5}).$$
Then 
$$|1-e^{\theta_1}|=|1-7^a\alpha^{-n}\sqrt{5}|<\frac{4}{\alpha^{n-m}}$$
by inequality (\ref{eq10}). The inequality
$$7^a=F_n-F_m\leqslant F_n-1<F_n+\frac{\beta^n}{\sqrt{5}}=\frac{\alpha^n}{\sqrt{5}}$$
implies that $7^a\alpha^{-n}\sqrt{5}<1$. Hence, we get $\theta_1<0$. In that case, since 
$\frac{4}{\alpha^{n-m}}<0.95=\frac{19}{20}$ for $n-m\geqslant 3$, then $|1-e^{\theta_1}|=1-e^{\theta_1}<\frac{19}{20}$, which implies $e^{\theta_1}>\frac{1}{20}$. From this, it follows that $e^{|\theta_1|}=e^{-\theta_1}<20$. Therefore, since $e^x-1>x$ for $x>0$, we obtain 
$$0<|\theta_1|<e^{|\theta_1|}-1=e^{|\theta_1|}|1-e^{\theta_1}|<\frac{80}{\alpha^{n-m}}$$
to yield $$0<|a\log(7)-n\log(\alpha)+\log(\sqrt{5})|<\frac{80}{\alpha^{n-m}}\cdot$$ 
If we divide this inequality by $\log(\alpha)$, we get
\begin{equation} \label{eq15}
0<\left|a\frac{\log(7)}{\log(\alpha)}-n+\frac{\log(\sqrt{5})}{\log(\alpha)}\right|<\frac{80}{\log(\alpha)}
\alpha^{-(n-m)}\leqslant 166.3\alpha^{-(n-m)}.
\end{equation}
Now, we try to apply Lemma \ref{lem1}. Set $\gamma:=\frac{\log(7)}{\log(\alpha)}\not\in \mathbb{Q}$, $\mu:=\frac{\log(\sqrt{5})}{\log(\alpha)}$, $A:=166.3$, $B:=\alpha$, and $\omega:=n-m$. Taking $M:=6.90212\cdot 10^{29}$, we found that $q_{68}$, the denominator of the $68$th convergent of $\gamma$, $q_{68}>6M$. Moreover, a quick calculation with Maple program gives
$$\varepsilon:=\|\mu q_{68}\|-M\|\gamma q_{68}\|=0.403101\ldots>0.$$
Thus, the inequality (\ref{eq15}) has no solution for 
$$n-m=\omega\geqslant \frac{\log(Aq_{68}/\varepsilon)}{\log(B)}=161.64334\ldots.$$
By Lemma \ref{lem1}, we get 
$$\frac{\log(Aq_{68}/\varepsilon)}{\log(B)}\leqslant \frac{\log(Aq_{68}/0.403101)}{\log(B)}\leqslant 161.64335.$$
A computer search with Maple program yields to $n-m\geqslant 161.64334$. So $n-m\leqslant 161$. Substituting  this upper bound for $n-m$ into inequality (\ref{eq14}) yielding $$n<3.16222\cdot 10^{14}(1+\log(n)),$$ which in turn yields $n<1.20246\cdot 10^{16}$.\\
We apply Lemma \ref{lem1} again, now to reduce the bound for $n$. Let
\begin{equation*}
\begin{split}
\theta_2:&=\log\left(7^a\alpha^{-n}\sqrt{5}\left(1-\alpha^{m-n}\right)^{-1}\right)\\
         &=a\log(7)-n\log(\alpha)+\log\left(\sqrt{5}\left(1-\alpha^{m-n}\right)^{-1}\right).
\end{split}
\end{equation*}  
In this case, from inequality (\ref{eq13}), we have
$$|1-e^{\theta_2}|<\frac{\sqrt{5}}{\alpha^n}<\frac{3}{\alpha^n}<\frac{1}{2}\cdot$$
If $\theta_2>0$, then $0<\theta_2<e^{\theta_2}-1=|1-e^{\theta_2}|<\frac{1}{2}$, so that 
$e^{|\theta_2|}=e^{\theta_2}<\frac{3}{2}<2$. If $\theta_2<0$, then $|1-e^{\theta_2}|=1-e^{\theta_2}<\frac{1}{2}$, which implies 
$e^{\theta_2}>\frac{1}{2}$. From this, we get $e^{|\theta_2|}=e^{-\theta_2}<2$. Altogether, 
$$0<|\theta_2|<e^{|\theta_2|}-1=e^{|\theta_2|}|1-e^{\theta_2}|<\frac{6}{\alpha^n}\cdot$$  
In both cases, we obtain $$0<|\theta_2|<\frac{6}{\alpha^n}$$
to yield $$0<|a\log(7)-n\log(\alpha)+\log(\sqrt{5}(1-\alpha^{m-n})^{-1})|<\frac{6}{\alpha^{n}}\cdot$$ 
If we divide both sides by $\log(\alpha)$, we have 
\begin{equation} \label{eq16}
0<\left|a\frac{\log(7)}{\log(\alpha)}-n+\frac{\log(\sqrt{5}(1-\alpha^{m-n})^{-1})}{\log(\alpha)}\right|<\frac{6}{\log(\alpha)}\alpha^{-n}\leqslant 13\alpha^{-n}.
\end{equation}
Set $\gamma:=\frac{\log(7)}{\log(\alpha)}$ and taking $M:=1.20246\cdot 10^{16}$, we found that $q_{41}$, the denominator of the $41$th convergent of $\gamma$, $q_{41}>6M$. Also, with 
$$\mu_{n-m}:=\frac{\log(\sqrt{5}(1-\alpha^{-(n-m)})^{-1})}{\log(\alpha)},\quad A:=13,\quad B:=\alpha,\quad \omega:=n,$$ and 
$$\varepsilon=\varepsilon(\mu_{n-m}):=\|\mu_{n-m}q_{41}\|-M\|\gamma q_{41}\|\in[0.00502\ldots,0.49335\ldots]$$
for all $n-m\in [3,161]$ with $n-m\neq 4$. A quick calculation using Maple program gives
$$0.00502\leqslant \varepsilon(\mu_{n-m})\leqslant 0.49335.$$
By Lemma \ref{lem1}, we conclude that inequality (\ref{eq16}) has no solution for
$$n=\omega \geqslant \frac{\log(Aq_{41}/\varepsilon)}{\log(B)}\geqslant \frac{\log(Aq_{41}/0.005)}{\log(B)}=106.26489\ldots$$ with $n-m \neq 4$. Since we assume that a solution has $n>200$, we conclude that there are no solutions with $n-m\neq 4$. If 
$n-m=4$, we have the only possible remaining solutions in this case
$$7^a=F_n-F_m=F_{m+4}-F_m=F_{m+3}+F_{m+2}-F_{m}=F_{m+3}+F_{m+1}=L_{m+2}$$
by the identity $F_{\ell+1}+F_{\ell-1}=L_{\ell}$. By Theorem \ref{thm2}, this case is not possible.
\end{proof}

\section{Solutions to the equation (\ref{eq1}) for $p=13$}
\begin{proof}[Proof of Theorem \ref{thm4}]
Assume that equation $F_n - F_m=13^a$. Using the identity (\ref{eq8}), we obtain the inequality
\begin{equation} \label{eqt9}
13^a=F_n - F_m<F_n \leqslant \alpha^{n-1}< \alpha^{n}.
\end{equation}
This shows that $a<n$.\\
Suppose $n-m=1$. Then $F_{m-1}=13^a$. By Theorem \ref{thm2}, we have $(n,m,a) \in \{(3,2,0),(4,3,0)\}$. Equally, in the case 
$n-m = 2$, then we get $F_{m+1}=13^a$. We have $(n,m,a) \in \{(2,0,0),(3,1,0),(8,6,1)\}$. Let $m = 0$. By Theorem \ref{thm2}, we have that solutions $F_2 - F_0 = 1$ and $F_7 - F_0 = 13$, then $(n,m,a)=(7,0,1)$. Assume that $m \geqslant 1$ and $n-m \geqslant 3$. From $1 \leqslant m < n \leqslant 200$, a direct calculation in Maple program, we obtain the only solution is 
$(n,m,a)=(9,8,1)$.\\
Assume then that $n>200$, $m \geqslant 1$ and $n-m \geqslant 3$. The formula (\ref{eq7}) can be rewritten as
$$\frac{\alpha^{n}}{\sqrt{5}} - 13^a = F_m + \frac{\beta^{n}}{\sqrt{5}},$$
we obtain
$$\left| \frac{\alpha^{n}}{\sqrt{5}} - 13^a \right| = \left| F_m +  \frac{\beta^{n}}{\sqrt{5}} \right| \leqslant F_m + \frac{\mid \beta \mid^{n}}{\sqrt{5}} < \alpha^m + \frac{1}{2}\cdot$$
Dividing both sides of the above inequality by $\frac{\alpha^{n}}{\sqrt{5}}$ and taking into account that $n>m$, we have 
\begin{equation} \label{eqt10}
\begin{split}
\left|1-13^a \alpha^{-n}\sqrt{5} \right|&<\alpha^{m-n}\sqrt{5} + \frac{\alpha^{-n}\sqrt{5}}{2}\\
                                      &=\alpha^{m-n}\sqrt{5}\left(1 + \frac{\alpha^{-m}}{2} \right)< \frac{4}{\alpha^{n-m}}\cdot
\end{split}                                        
\end{equation}
Now we apply Theorem \ref{thm1}, we take the parameters $t:=3$ and $\eta_1 := 13$, $\eta_2 := \alpha$, $\eta_3 := \sqrt{5}$. We also take $b_1 := a$, $b_2 := -n$ and $b_3 := 1$. Note that the three numbers $\eta_1$, $\eta_2$, and $\eta_3$ are positive real numbers and elements of the field $\mathbb{K}:=\mathbb{Q}(\sqrt{5})$, so we can take $D:=[\mathbb{K}:\mathbb{Q}]=2$. We show that $\Gamma_1 := 13^a\alpha^{-n}\sqrt{5} - 1 \neq 0$. Assume that $\Gamma_1 = 0$. We obtain $\alpha^{2n} = 5.13^{2a} \in \mathbb{Q}$, which is impossible since $\alpha^{2n} \not \in \mathbb{Q}$. Moreover, since $h(\eta_1)=\log(13)= 2.56494\ldots$, we take $A_1:=5.15$. Since $h(\eta_2)=\frac{\log(\alpha)}{2}=0.2406\ldots$ and $h(\eta_3)=\log(\sqrt{5})=0.80471\ldots$, we can take $A_2:=0.5$ and $A_3:=1.65$. Since $a<n$, it follows that
$$B:= \max \{|b_1|,|b_2|,|b_3| \} = \max \{|a|,|-n|,1 \} = n.$$
Hence, by inequality (\ref{eqt10}) and using Theorem \ref{thm1}, we have
$$\exp\left(-C(1+ \log(2))(1+ \log(n))(5.15\cdot 0.5\cdot 1.65)\right)<|\Gamma_1|<\frac{4}{\alpha^{n-m}}$$
and so 
$$(n-m)\log(\alpha) - \log(4) < C\left(1+ \log(2)\right)\left(1+ \log(n)\right)\left(5.15\cdot 0.5\cdot 1.65\right),$$
where $C=1.4\cdot 30^6 \cdot 3^{4.5}\cdot 2^2$.\\
From the last inequality, we obtain 
\begin{equation} \label{eqt11}
(n-m)\log(\alpha) - \log(4)<4.12019\cdot 10^{12}\left(1+ \log(n)\right).
\end{equation}
We again rewrite equation (\ref{eq1}) as
$$\frac{\alpha^n}{\sqrt{5}}-\frac{\alpha^m}{\sqrt{5}}-13^a=\frac{\beta^n}{\sqrt{5}}-\frac{\beta^m}{\sqrt{5}}$$  
and taking absolue values, we obtain
$$\left|\frac{\alpha^n(1-\alpha^{m-n})}{\sqrt{5}}-13^a\right|=\frac{|\beta|^n+|\beta|^m}{\sqrt{5}}<\frac{1}{3},$$
where we used the fact that $|\beta|^n+|\beta|^m<\frac{2}{3}$ for $n>200$. If we divide both sides of the above inequality by  
$\frac{\alpha^n(1-\alpha^{m-n})}{\sqrt{5}}$, we get
\begin{equation} \label{eqt12}
\left|1-13^a\alpha^{-n}\sqrt{5}\left(1-\alpha^{m-n}\right)^{-1}\right|<\frac{\alpha^{-n}\sqrt{5}\left(1-\alpha^{m-n}\right)^{-1}}{3}\cdot
\end{equation}
Since $\alpha^{m-n}=\frac{1}{\alpha^{n-m}}<\frac{1}{\alpha}<\frac{2}{3}$, we see that $1-\alpha^{m-n}>\frac{1}{3}$ and hence 
$\left(1-\alpha^{m-n}\right)^{-1}<3$. Then from (\ref{eqt12}), it follows that  
\begin{equation} \label{eqt13}
\left|1-13^a\alpha^{-n}\sqrt{5}\left(1-\alpha^{m-n}\right)^{-1}\right|<\frac{\sqrt{5}}{\alpha^{n}}<\frac{3}{\alpha^{n}}\cdot
\end{equation}
We apply Theorem \ref{thm1} once again. Take $\eta_1:=13$, $\eta_2:=\alpha$, $\eta_3:=\sqrt{5}(1-\alpha^{m-n})^{-1}$, $b_1:=a$, $b_2:=-n$ and $b_3:=1$. The positive real numbers $\eta_1$, $\eta_2$, and $\eta_3$ lie in $\mathbb{K}:=\mathbb{Q}(\sqrt{5})$, so we have $D=2$. Put $\Gamma_2:=13^a\alpha^{-n}\sqrt{5}(1- \alpha^{m-n})^{-1}-1$. We show that $\Gamma_2$ is not zero. Because if $\Gamma_2=0$, then we get
$$\frac{\alpha^n}{\sqrt{5}}-\frac{\alpha^m}{\sqrt{5}}=13^a=F_n-F_m=\frac{\alpha^n}{\sqrt{5}}-\frac{\alpha^m}{\sqrt{5}}+\frac{\beta^m}{\sqrt{5}}-\frac{\beta^n}{\sqrt{5}},$$
which implies that $\beta^n=\beta^m$. Therefore, this is not acceptable since $n>m$. Similarly, by equation (\ref{eq3}), we have $h(\eta_1)=\log(13)=2.56494\ldots$, and $h(\eta_2)=\frac{\log(\alpha)}{2}=0.2406\ldots$, we take $A_1:=5.15$ and $A_2:=0.5$.
Bisides, by equations (\ref{eq4}), (\ref{eq5}), and (\ref{eq6}), we get that  
\begin{center}
$h(\eta_3)\leqslant \log(\sqrt{5})+|m-n|\frac{\log(\alpha)}{2}+\log(2)=\frac{1}{2}\left(\log(20)+(n-m)\log(\alpha)\right)$.
\end{center}
A simple calculation shows that $|\log(\eta_3)|<\log(5)+(n-m)\log(\alpha)$, and so we can take $A_3:=\log(20)+(n-m)\log(\alpha)$. 
If follows again that $B:=\max\{|a|,|-n|,1\}=n$, since $a<n$. Thus, from inequality (\ref{eqt13}) and Theorem \ref{thm1}, we obtain
\begin{multline*}
\exp\left(-C(1+\log(2))(1+ \log(n))(5.15\cdot 0.5(\log(20)+(n-m)\log(\alpha)))\right)\\
<|\Gamma_2|<\frac{\sqrt{5}}{\alpha^{n}}<\frac{3}{\alpha^{n}}
\end{multline*}
with $C=1.4\cdot 30^6\cdot3^{4.5}\cdot2^2$, which yields
\begin{equation}\label{eqt14}
\log(\frac{\alpha^n}{3})<2.49708\cdot 10^{12}\left(1+ \log(n)\right)\left(\log(20)+(n-m)\log(\alpha)\right).
\end{equation} 
We can now substitute inequality (\ref{eqt11}), we obtain
\begin{equation*}
\begin{split}
n\log(\alpha)-\log(3)<&2.49708\cdot 10^{12}\left(1+ \log(n)\right)\\
                                                &\cdot\left(\log(20)+4.12019\cdot 10^{12}(1+ \log(n))+\log(4)\right)                                                                                    
\end{split}
\end{equation*}
and so
\begin{equation*}
n<2.13803\cdot10^{25}+4.27606\cdot10^{25}\log(n)+2.13803\cdot10^{25}(\log(n))^2.
\end{equation*}
Then $$n<1.55331\cdot10^{26}\left(\log(n)\right)^2.$$
With the help of Maple program, it is seen that $$n<7.34589\cdot10^{29}.$$
Let us try to reduce the upper bound on $n$ by applying Lemma \ref{lem1} two times. Suppose  
$$\theta_1:=\log(13^a\alpha^{-n}\sqrt{5})=a\log(13)-n\log(\alpha)+\log(\sqrt{5}).$$
Then 
$$|1-e^{\theta_1}|=|1-13^a\alpha^{-n}\sqrt{5}|<\frac{4}{\alpha^{n-m}}$$
by inequality (\ref{eqt10}). The inequality
$$13^a=F_n-F_m\leqslant F_n-1<F_n+\frac{\beta^n}{\sqrt{5}}=\frac{\alpha^n}{\sqrt{5}}$$
implies that $13^a\alpha^{-n}\sqrt{5}<1$. Hence, we get $\theta_1<0$. In that case, since 
$\frac{4}{\alpha^{n-m}}<0.95=\frac{19}{20}$ for $n-m\geqslant 3$, then $|1-e^{\theta_1}|=1-e^{\theta_1}<\frac{19}{20}$, which implies $e^{\theta_1}>\frac{1}{20}$. From this, it follows that $e^{|\theta_1|}=e^{-\theta_1}<20$. Therefore, since $e^x-1>x$ for $x>0$, we obtain 
$$0<|\theta_1|<e^{|\theta_1|}-1=e^{|\theta_1|}|1-e^{\theta_1}|<\frac{80}{\alpha^{n-m}}$$
to yield $$0<|a\log(13)-n\log(\alpha)+\log(\sqrt{5})|<\frac{80}{\alpha^{n-m}}\cdot$$ 
If we divide this inequality by $\log(\alpha)$, we get
\begin{equation} \label{eqt15}
0<\left|a\frac{\log(13)}{\log(\alpha)}-n+\frac{\log(\sqrt{5})}{\log(\alpha)}\right|<\frac{80}{\log(\alpha)}
\alpha^{-(n-m)}\leqslant 166.3\alpha^{-(n-m)}.
\end{equation}
Now, we try to apply Lemma \ref{lem1}. Set $\gamma:=\frac{\log(13)}{\log(\alpha)}\not\in \mathbb{Q}$, $\mu:=\frac{\log(\sqrt{5})}{\log(\alpha)}$, $A:=166.3$, $B:=\alpha$, and $\omega:=n-m$. Taking $M:=7.34589\cdot 10^{29}$, we found that $q_{64}$, the denominator of the $64$th convergent of $\gamma$, $q_{64}$, exceeds $6M$. Moreover, a quick calculation with Maple program gives
$$\varepsilon:=\|\mu q_{64}\|-M\|\gamma q_{64}\|=0.29693\ldots>0.$$
Thus, the inequality (\ref{eqt15}) has no solution for 
$$n-m=\omega\geqslant \frac{\log(Aq_{64}/\varepsilon)}{\log(B)}=160.74358\ldots.$$
By Lemma \ref{lem1}, we get 
$$\frac{\log(Aq_{64}/\varepsilon)}{\log(B)}\leqslant \frac{\log(Aq_{64}/0.29693)}{\log(B)}\leqslant 160.74360.$$
A computer search with Maple program yields to $n-m\geqslant 160.74358$. So $n-m\leqslant 160$. Substituting  this upper bound for $n-m$ into inequality (\ref{eqt14}) yielding $$n<4.15078\cdot 10^{14}(1+\log(n)),$$ which in turn yields $n<1.58996\cdot 10^{16}$.\\
We apply Lemma \ref{lem1} again, now to reduce the bound for $n$. Let
\begin{equation*}
\begin{split}
\theta_2:&=\log\left(13^a\alpha^{-n}\sqrt{5}\left(1-\alpha^{m-n}\right)^{-1}\right)\\
         &=a\log(13)-n\log(\alpha)+\log\left(\sqrt{5}\left(1-\alpha^{m-n}\right)^{-1}\right).
\end{split}
\end{equation*}   
In this case, from inequality (\ref{eqt13}), we have
$$|1-e^{\theta_2}|<\frac{\sqrt{5}}{\alpha^n}<\frac{3}{\alpha^n}<\frac{1}{2}\cdot$$
If $\theta_2>0$, then $0<\theta_2<e^{\theta_2}-1=|1-e^{\theta_2}|<\frac{1}{2}$, so that 
$e^{|\theta_2|}=e^{\theta_2}<\frac{3}{2}<2$. If $\theta_2<0$, then $|1-e^{\theta_2}|=1-e^{\theta_2}<\frac{1}{2}$, which implies 
$e^{\theta_2}>\frac{1}{2}$. From this, we get $e^{|\theta_2|}=e^{-\theta_2}<2$. Altogether, 
$$0<|\theta_2|<e^{|\theta_2|}-1=e^{|\theta_2|}|1-e^{\theta_2}|<\frac{6}{\alpha^n}\cdot$$  
In both cases, we obtain $$0<|\theta_2|<\frac{6}{\alpha^n}$$
and so $$0<|a\log(13)-n\log(\alpha)+\log(\sqrt{5}(1-\alpha^{m-n})^{-1})|<\frac{6}{\alpha^{n}}\cdot$$ 
If we divide both sides by $\log(\alpha)$, we have 
\begin{equation} \label{eqt16}
0<\left|a\frac{\log(13)}{\log(\alpha)}-n+\frac{\log(\sqrt{5}(1-\alpha^{m-n})^{-1})}{\log(\alpha)}\right|<\frac{6}{\log(\alpha)}\alpha^{-n}\leqslant 13\alpha^{-n}.
\end{equation}
Set $\gamma:=\frac{\log(13)}{\log(\alpha)}$ and taking $M:=1.58996\cdot 10^{16}$, we found that $q_{33}$, the denominator of the $33$th convergent of $\gamma$, $q_{33}$, exceeds $6M$. Also, with 
$$\mu_{n-m}:=\frac{\log(\sqrt{5}(1-\alpha^{-(n-m)})^{-1})}{\log(\alpha)},\quad A:=13,\quad B:=\alpha,\quad \omega:=n,$$ and 
$$\varepsilon=\varepsilon(\mu_{n-m}):=\|\mu_{n-m}q_{33}\|-M\|\gamma q_{33}\|\in[0.00412\ldots,0.45321\ldots]$$
for all $n-m\in [3,160]$ with $n-m\neq 4,66,88$. A quick calculation using Maple program gives
$$0.00412\leqslant \varepsilon(\mu_{n-m})\leqslant 0.45321.$$
By Lemma \ref{lem1}, we conclude that inequality (\ref{eqt16}) has no solution for
$$n=\omega \geqslant \frac{\log(Aq_{33}/\varepsilon)}{\log(B)}\geqslant \frac{\log(Aq_{33}/0.004)}{\log(B)}=100.22156\ldots$$ with $n-m \neq 4,66,88$. Since we assume here that $n>200$, we conclude that there are no solutions where $n-m \neq 4,66,88$. Now consider the cases $n-m=4,66,88$. If $n-m=4$, then the only possible remaining solutions in case yields
$$13^a=F_n-F_m=F_{m+4}-F_m=F_{m+3}+F_{m+2}-F_{m}=F_{m+3}+F_{m+1}=L_{m+2}$$
by the identity $F_{\ell+1}+F_{\ell-1}=L_{\ell}$. By Theorem \ref{thm2}, this case is impossible.
If $n-m=66$, then we have the equation $13^a=F_n-F_m=F_{m+33}L_{33}$ by Lemma \ref{lem2}, which is not possible. Lastly, assume that $n-m=88$. Then the equation $13^a=F_n-F_m=F_{44}L_{m+44}$ by Lemma \ref{lem2}, which is not possible.   
\end{proof}

\end{document}